\documentclass{elsart}

\usepackage{amssymb}
\usepackage{latexsym}
\usepackage{amsmath}
\usepackage{amsfonts}
\usepackage{graphicx}
\usepackage{subfig}
\usepackage{color}
\usepackage{times}

\usepackage[latin1]{inputenc}
\usepackage[spanish]{babel}

\newtheorem{definition}{\bf Definition}[section]
\newtheorem{Lem}[definition]{\bf Lemma}
\newtheorem{Obs}[definition]{\bf Observation}
\newtheorem{Thm}[definition]{\bf Theorem}
\newtheorem{Prop}[definition]{\bf Proposition}
\newtheorem{Cor}[definition]{\bf Corollary}

\newtheorem{Rem}[definition]{\bf Remark}
\newenvironment{proof}{\noindent \textit{Proof. }}{\hfill\hfill\qed}

\def\Q#1{{\operatorname{Q}(#1)}}
\def\S#1{{\operatorname{S}(#1)}}

\begin{document}

\begin{frontmatter}

\title{The differential on Graph Operator $\S{G}$.}

\author{Gerardo Reyna}
\address{Faculty of Mathematics, Autonomous University of Guerrero.
	Carlos E. Adame 5, Col. La Garita, Acapulco, Guerrero, Mexico}
\ead{gerardoreynah@hotmail.com }

\author{Jair Castro Simon}
\address{Faculty of Mathematics, Autonomous University of Guerrero.
Carlos E. Adame 5, Col. La Garita, Acapulco, Guerrero, Mexico}
\ead{castrosimonjair@gmail.com}

\author{Omar Rosario}
\address{Faculty of Mathematics, Autonomous University of Guerrero.
	Carlos E. Adame 5, Col. La Garita, Acapulco, Guerrero, Mexico}
\ead{omarrosarioc@gmail.com}


\date{\today}
\maketitle

\begin{abstract}
Let $G=(V(G),E(G))$ be a simple graph with vertex set $V(G)$ and edge set
$E(G)$. Let $S$ be a subset of $V(G)$, and let $B(S)$ be the set of
neighbours of $S$ in $V(G) \setminus S$. The  differential
$\partial(S)$ of $S$ is defined as $|B(S)|-|S|$. The maximum value
of $\partial(S)$ taken over all subsets $S\subseteq V$ is the 
differential $\partial(G)$ of $G$. A graph operator is a mapping $F: G\rightarrow G'$, 
where $G$ and $G'$ are families of graphs.
The graph $\S{G}$ is defined as the graph obtained from $G$ con bipartici\'on de v\'ertices $V(G)\cup E(G)$, donde hay tantas aristas entre $v \in V(G)$ y $e \in E(G)$, como veces $e$ sea incidente con $v$ en $G$.   
In this paper we study the relationship between $\partial(G)$ and $\partial(\S{G})$. Besides, we relate the differential of a graph with known parameters of a graph, namely, its domination  and independence number.
\end{abstract}

\end{frontmatter}
\vspace{0.1cm} \small{ {\it Keywords:}  Differential of a graph;
Operators Graphs; Differential.

}

\section{Introduction}
Social networks, such as Facebook or Twitter, have served as an
important medium for communication and information disseminating. As
a result of their massive popularity, social networks now have a wide variety
of applications in the viral marketing of products and political
campaigns. Motivated by its numerous applications, some
authors \cite{KeKlTa1,KeKlTa2} have proposed several influential maximization problems,
 which share a fundamental algorithmic problem for information diffusion in social networks: the problem of
 determining the best group of nodes to influence the rest.
As it was showed in \cite{BeFe}, the study of  the differential of a
graph $G$, could be motivated from such scenarios.

The study of $\partial(G)$ together with a variety of other kinds of
differentials of a set, started in \cite{MaHaHeHeSl}. In particular,
several bounds for $\partial(G)$ were given. The differential of a
graph has also been investigated in
\cite{BaBeSi,Be,BeDeMaSi,BeFe,BeFe2,BeFe3,BeFeSi,BeRoSi,RoYo,Si},
and it was proved in \cite{BeFeSi} that $\partial(G)+\gamma_R(G)=n$,
where $n$ is the order of the graph $G$ and $\gamma_R(G)$ is the
Roman domination number of $G$, so every bound for the differential
of a graph can be used to get a bound for the Roman domination
number. The differential of a set $D$ was also considered in
\cite{GoHe}, where it was denoted by $\eta(D)$, and the minimum
differential of an independent set was considered in \cite{Zh}. The
case of the $\beta$-differential of a graph or enclaveless number,
defined as $\psi(G):= \max\{ |B(D)|:D \subseteq V(G)\}$, was studied in
\cite{BaBeLeSi,Sl1}. As usual, we will denote by $E_n$ the graph of order $n$ with no edges.

Throughout this paper, $G=(V(G),E(G))$ is a simple graph of order $n\geq 2$ with
vertex set $V(G)$ and edge set $E(G)$. Let $u$ and $v$ be two distinct vertices of $V(G)$, and let $S$ be a
subset of $V(G)$. As usual, $N(v)$ is the set of neighbours that $v$ has
in $V(G)$, and $N[v]$ is the closed neighbourhood of $v$, i.e., $N[v] := N(v)\cup \{v\}$. 
We denote by $\delta(v) := |N(v)|$ the degree of $v$ in $G$, and by $\delta(G)$ and $\Delta(G)$ the
 minimum and the maximum degree of $G$, respectively.
The subgraph of $G$ induced by $S$ will be denoted by $G[S]$, and
the complement of $S$ in $V(G)$ by $\overline{S}$. Then
$N_{\overline{S}}(v)$ is the set of neighbours that $v$ has in
$\overline{S}=V \setminus S$.  We let $N(S):=\displaystyle{\bigcup_{v\in S}}N(v)$
and $N[S]:=N(S)\cup S$. An \emph{external private neighbour of} $v\in S$
{\em with respect to} $S$ is a vertex $w\in N(v)\cap \overline{S}$ such
that $w\notin N(u)$ for every $u\in S\setminus\{v\}$. The set of all
external private neighbours of $v$ with respect to $S$ is denoted by epn$[v,S]$.
 Let $B(S)$ be the set of vertices in $\overline{S}$ that have a neighbour in $S$, and
let $C(S)$ be the set $\overline{S\cup B(S)}$. Then $\{S,B(S),C(S)\}$ son conjuntos disjuntos 
entre s\'i tales que $V(G)=S\cup B(S)\cup C(S)$. The {\em differential of a set} $S$ is defined as 
$\partial(S)=|B(S)|-|S|$ and the {\em differential of a graph} $G$ is 
defined as $\partial(G)=max \{ \partial(S): S\subseteq V \}$. We will 
say that $S \subseteq V$ is a {\em differential set} if $\partial(S)=\partial(G)$, $S$ 
is a \emph{minimum (maximum) differential set} if it has minimum (maximum) cardinality 
among all differential sets.

Notice that if $G$ is disconnected, and $G_1,\ldots ,G_k$ are its 
connected components, then $\partial(G)=\partial(G_1)+\cdot\cdot\cdot+\partial(G_k)$. 
In view of this, from now on we only consider connected graphs.


We recall that a {\em vertex cover} of a graph $G$  is a subset $C\subseteq V(G)$ such that every edge of $G$ has at least one end vertex in $C$.  The  {\em vertex-covering number} of $G$ is the size of any smallest vertex cover in $G$ and is denoted by $\tau(G)$. A subset  $D\subseteq V(G)$ is a {\em dominating set} of $G$ if each vertex in $V(G)$ is in $D$ or is adjacent to a vertex in $D$. The {\em domination number} of $G$ is the minimum cardinality of a dominating set of $G$ and is denoted by $\gamma(G)$. A subset  $I\subseteq V(G)$ is an {\em independent set} of $G$ if each any two distinct vertices in $I$ are not adjacent in $G$. The {\em independence number} of $G$ is the maximum cardinality of an independent set of $G$ and is denoted by $\alpha(G)$. A subset $M \subseteq E(G)$ is a {\em matching} de $G$ if each any two distinct edge in $M$ are not incidents in $G$. The {\em matching number} of $G$ is the maximum cardinality of a matching of $G$ and is denoted by $\beta(G)$. Finally, a {\em edge cover} of $G$ is a subset $A \subseteq E(G)$ such that every vertex of $G$ is incident to at least one edge of the set
$A$. The {\em edge-covering number} of $G$ is the size of a minimum edge covering in $G$ and is denoted by $\rho(G)$.

The operator $\S{G}$ (following the notations \cite{CvDoSa}) is defined as the graph obtained from $G$ 
con bipartici\'on de v\'ertices $V(G)\cup E(G)$, donde hay tantas aristas entre $v \in V(G)$ y $e \in E(G)$, como veces $e$ sea incidente con $v$ en $G$.(see Figure \ref{Fig:S(G)}). 

\begin{figure}
\begin{center}
\includegraphics[width=10cm]{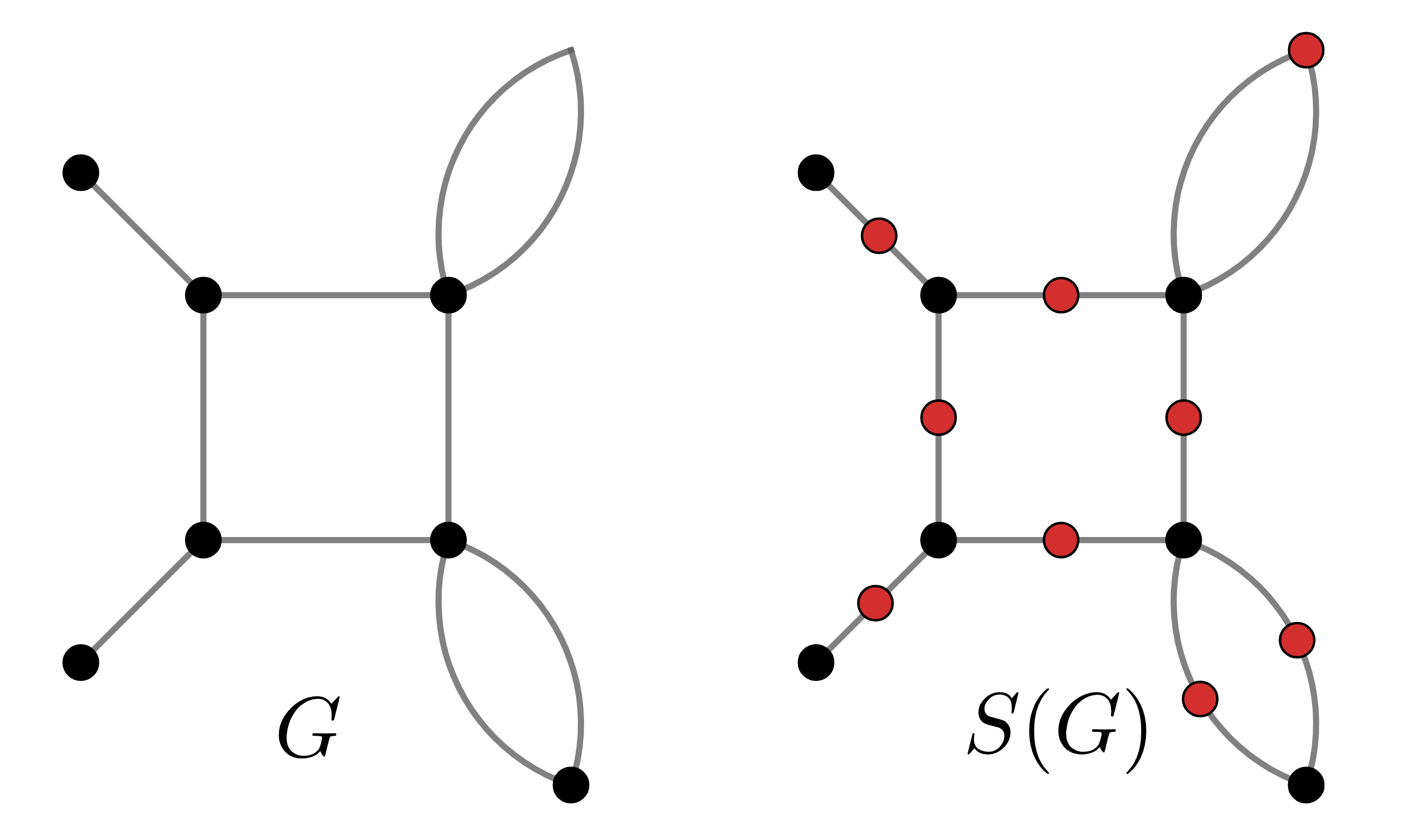} 
\end{center}
\caption{En la izquierda la gr\'afica $G$ tiene un lazo y aristas m\'ultiples, en la derecha la representaci\'on de la gr\'afica subdivisi\'on $\S{G}$.}
\label{Fig:S(G)}
\end{figure}

Para gr\'aficas simples se puede dar un definici\'on de la operaci\'on subdivisi\'on $\S{G}$ como sigue: \emph{La gr\'afica \textbf{subdivisi\'on} $\S{G}$ es obtenida de $G$ insertando un v\'ertice adicional en cada arista de $G$.} (see Figure \ref{Fig:Q(G)}). 

\begin{figure}
\begin{center}
\includegraphics[width=10cm]{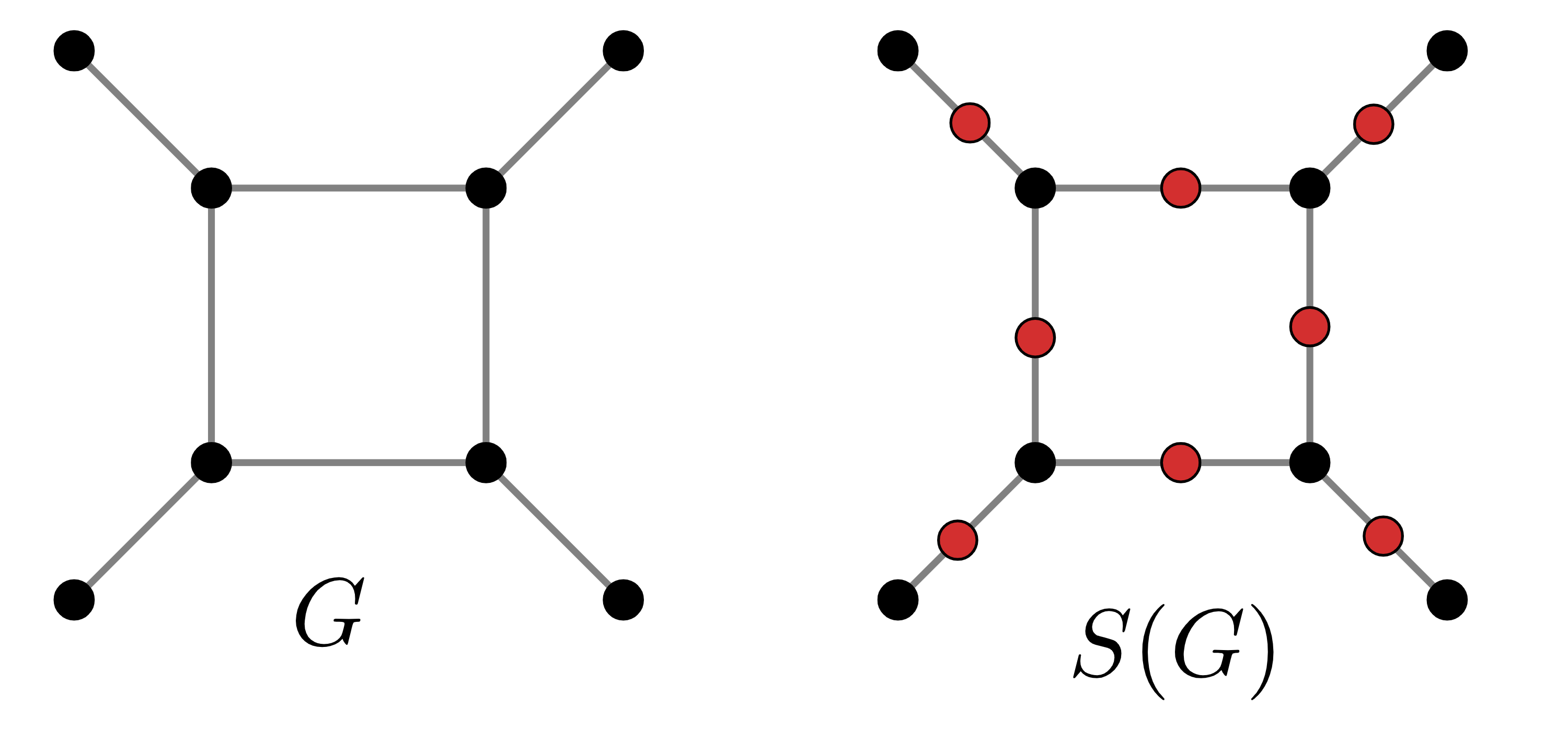} 
\end{center}
\caption{The graph on the right is the corresponding $\S{G}$ of the graph $G$ on the left.}
\label{Fig:Q(G)}
\end{figure}

\section{The differential of $\S{G}$.}

We start by listing some basic properties of $\S{G}$, which can be deduced easily from the
definition of $\S{G}$. We use $E_n$ to denote the graph of order $n \geq 2$ with no edges.

\begin{Rem}\label{rem-Q(G)}
\begin{itemize}
\item[i)]Al conjunto de los v\'ertices de $\S{G}$ correspondientes a las aristas de $G$ los vamos a denotar por $S$, y al resto por $V$, adem\'as $|V|=|V(G)|$ y $|S|=|E(G)|$.
\item[ii)]Los conjuntos $S$ y $V$ son conjuntos independientes.
\end{itemize}
\end{Rem}

\begin{Prop}\label{s1} 
Let $G$ be a graph of order $n$, then
\noindent
\begin{enumerate}
\item[i)] $|E(\S{G})|=2 |S|$,
\item[ii)] $G\cong \S{G}$ if and only if  $G\cong E_n$,
\item[iii)] $\delta(w)=2$, for all $w \in S$,
\item[iv)] $\delta_{\S{G}}(v)=\delta_G(v)$, for all $v \in V$.
\end{enumerate} 
\end{Prop}

\begin{Thm}\cite{Galai} \label{Galai}
Para cualquier gr\'afica $G$,
\begin{itemize}
\item[i)]  $|V(G)|=\alpha(G)+\tau(G)$,
\item[ii)] $|V(G)|=\beta(G)+\rho(G)$.
\end{itemize}
\end{Thm}

\begin{Obs}\label{arbol}
Si $T$ es un \'arbol, entonces $|E(T)|=|V(T)|-1$.
\end{Obs}

\begin{Thm}\label{geras}
Si $G$ es una gr\'afica conexa de orden $n$ y tama\~no $m$, entonces $\alpha(\S{G})=m\acute{a}x\lbrace n,m\rbrace$.
\end{Thm}
\begin{proof}
La demostraci\'on se har\'a en dos pasos.
\begin{enumerate}
\item[Paso 1:]\emph{El resultado es cierto cuando $G$ es un \'arbol.}\\
Usaremos inducci\'on sobre $n$ siendo el paso base para $n=2$, no muy dif\'icil de verificar.\\
Dado un \'arbol $G$ de orden $n>2$, sea $I$ un conjunto independiente de $\S{G}$ de tama\~no m\'aximo. Ya que $S$ es un conjunto independiente maximal en $\S{G}$, por el Lema \ref{arbol} $m<n$, debe existir $e \in S \setminus I$. De esta forma se tiene que 
$$\S{G}\setminus \lbrace e \rbrace = \S{G \setminus \lbrace e \rbrace} = \S{G_1} \cup \S{G_2} $$ donde $G_1$ y $G_2$ son las componentes de $G\setminus \lbrace e \rbrace$. Notemos que $I$ sigue siendo un conjunto independiente en $\S{G\setminus \lbrace e \rbrace}$, y por hip\'otesis de inducci\'on se tiene que
$$\alpha(\S{G\setminus \lbrace e \rbrace})=\alpha(\S{G_1})+\alpha(\S{G_2})=n_1+n_2=n=m\acute{a}x \lbrace n,m\rbrace.$$ Y de aqu\'i obtenemos que $|I|=m\acute{a}x \lbrace n,m\rbrace$.
\item[Paso 2:] \emph{El resultado es cierto en general.}\\
Este paso lo demostraremos por inducci\'on sobre $m$ siendo el paso base precisamente el Paso $1$.\\
Sea $I \subseteq V(\S{G})$ un conjunto independiente m\'aximo. Si $S \subseteq I$ se tiene que $\alpha(\S{G})=|S|$ we are done. En otro caso, consid\'erese la gr\'afica $\S{G}\setminus \lbrace e \rbrace$, donde $e \in S\setminus I$; n\'otese que $I$ es un conjunto independiente en esta gr\'afica. Si $\S{G} \setminus \lbrace e \rbrace$ es conexa, tomando en cuenta que $\S{G}\setminus \lbrace e \rbrace = \S{G \setminus \lbrace e \rbrace}$, la hip\'otesis de inducci\'on implica que $\alpha(\S{G}\setminus \lbrace e \rbrace)=m\acute{a}x \lbrace n, m-1 \rbrace$ y por lo tanto $|I|=m\acute{a}x\lbrace n,m \rbrace$. Si $\S{G}\setminus \lbrace e \rbrace$ es disconexa, entonces las dos componentes $G_1$ y $G_2$ de $G\setminus \lbrace e \rbrace$ son tales que $$\S{G}\setminus \lbrace e \rbrace=\S{G_1} \cup \S{G_2} $$ y una vez m\'as la hip\'otesis de inducci\'on y la independencia de $I$  en $\S{G}\setminus \lbrace e \rbrace$ nos da que $$\alpha(\S{G}\setminus \lbrace e \rbrace)=\alpha(\S{G_1})+\alpha(\S{G_2})=m\acute{a}x\lbrace n_1,m_1 \rbrace+m\acute{a}x\lbrace n_2,m_2 \rbrace \leq m\acute{a}x\lbrace n,m \rbrace $$ y de aqu\'i se obtiene que $|I|=m\acute{a}x\lbrace n_,m \rbrace$.
\end{enumerate}
\end{proof}

\begin{Cor}
Si $G$ es una gr\'afica conexa de orden $n$ y tama\~no $m$, entonces $\tau(\S{G})=min\lbrace n,m\rbrace$.
\end{Cor}

\begin{Lem}\label{mio}
Una gr\'afica $G$ contiene como subgr\'afica a lo m\'as a un $k$-ciclo si y s\'olo si $|E(G)|\leq |V(G)|$.
\end{Lem}


\begin{Thm}\label{eq1}
Sea $G$ una gr\'afica. Los siguientes enunciados son equivalentes.
\begin{enumerate}
\item[i)] $G$ contiene a lo m\'as un ciclo,
\item[ii)] $\alpha(\S{G})=|V|$,
\item[iii)] $\tau(\S{G})=|S|$.
\end{enumerate}
\end{Thm}
\begin{proof}

\emph{i)} $\Rightarrow$ \emph{ii)}. Supongamos que $G$ contiene a lo m\'as un ciclo, es decir $G$ no tiene ciclos \'o $G$ tiene un \'unico ciclo. Si $G$ no tiene ciclos, entonces $G$ es un \'arbol, y por lo tanto $|S|=|E(G)|=|V(G)|-1<|V(G)|=|V|$. Si $G$ contiene un \'unico $k$-ciclo, entonces $|V(G)|=|E(G)|$, en cualquier caso, el resultado se desprende del Teorema \ref{geras}.

\emph{ii)} $\Rightarrow$ \emph{i)}. Supongamos que $\alpha(\S{G})=|V|$. Ya que $S$ y $V$ son conjuntos independientes, $|E(G)|=|S|\leq |V|=|V(G)|$. Por el Lema \ref{mio} $G$ contiene a lo m\'as un ciclo.

Las implicaciones \emph{ii)} $\Rightarrow$ \emph{iii)} y \emph{iii)} $\Rightarrow$ \emph{ii)} se desprenden del Teorema \ref{Galai}.
\end{proof}

N\'otese que en virtud del Teorema \ref{geras}, el teorema anterior, nos da de manera impl\'icita una caracterizaci\'on de aquellas gr\'aficas $G$, para las cuales $\alpha(\S{G})=|S|$. Tomando en cuenta adem\'as, que $|V(G)|=\alpha(G)+\tau(G)$ (Teorema \ref{Galai}) se tienen el siguiente corolario.

\begin{Cor}\label{alf=S}
Sea $G$ una gr\'afica. Los siguientes enunciados son equivalentes.
\begin{enumerate}
\item[i)] $G$ contiene al menos un ciclo,
\item[ii)] $\alpha(\S{G})=|S|$,
\item[iii)] $\tau(\S{G})=|V|$.
\end{enumerate}
\end{Cor}

\begin{Prop}
Sea $G$ una gr\'afica, para cada conjunto independiente m\'aximo $I$ de $\S{G}$ se cumple que $I=V$ \'o $I=S$ si y s\'olo si $G$ es un \'arbol \'o $\delta(G)\geq 2$. 
\end{Prop}
\begin{proof}
Supongamos que $G$ no es un \'arbol y que $\delta(G)=1$. 
Entonces $G$ contiene al menos un ciclo, por el Corolario \ref{alf=S}, $\alpha(\S{G})=|S|$. Entonces existe un conjunto independiente m\'aximo $I \subset V(\S{G})$ tal que $|I|=|S|$. Ya que el grado m\'inimo de $G$ es $1$, existe un v\'ertice $v \in V$ tal que $\delta(v)=1$. Sea $e \in S$ tal que $v \in N_{\S{G}}(e)$. El conjunto $I=(S\setminus \lbrace e \rbrace)\cup \lbrace v \rbrace$ es un conjunto independiente m\'aximo de $\S{G}$ que tiene elementos de $V$ y de $S$.  

Supongamos que $G$ es un \'arbol. Por el Teorema \ref{eq1}, $\alpha(\S{G})=n$. Entonces existe un conjunto independiente m\'aximo $I \subset V(\S{G})$ tal que $|I|=n$. Vamos a demostrar que $I=V$. Supongamos que $I=A \cup B$ con $A \subseteq V$ y $B \subseteq S$ tal que $|A|=r\not = 0$ y $|B|=n-r\not =0$. Obs\'ervese que $N_{\S{G}}(B)\cap A=\emptyset$. Adem\'as $B \subseteq E(\langle N_{\S{G}}(B) \rangle_G)$. Ya que $G$ es un \'arbol, entonces $\langle N_{\S{G}}(B)\rangle_G$ es un bosque y por lo tanto $n-r=|B|\leq |E(\langle N_{\S{G}}(B)\rangle_G)|<|N_{\S{G}}(B)|$. Por lo tanto $|N_{\S{G}}(B)\cup A|=|N_{\S{G}}(B)|+|A|=|N_{\S{G}}(B)|+r>(n-r)+r=n$, lo cual es una contradicci\'on. Por lo tanto $I=V$. 

Ahora supongamos que $\delta(G)\geq 2$.
\begin{enumerate}
\item[Paso 1:] \emph{Si $A \subseteq V$ entonces $|A|\leq |N_{\S{G}}(A)|$. Adem\'as, si se da la igualdad entonces $G$ es un ciclo.}

Sea $H=\langle N_{\S{G}}(A) \rangle_{G}$ la subgr\'afica de $G$ inducida por las aristas del conjunto $N_{\S{G}}(A)$. Observemos que $A \subseteq V(H)$ y que para cada $x \in A$, $\delta_{H}(x)=\delta_{G}(x)$. Por lo tanto $$2|A|\leq \sum_{x \in A} \delta_G(x) \leq \sum_{x \in A} \delta_H(x) \leq \sum_{x \in V(H)} \delta_H(x)=2|E(H)|=2|N_{\S{G}}(A)|,$$
de aqu\'i se desprende la desigualdad deseada.\\ 
Si se da la igualdad $|A|=|N_{\S{G}}(A)|$, entonces en particular se tiene que $$\displaystyle{\sum_{x \in A}} \delta_H(x)=2|A|,$$ de lo cual se obtiene que $\delta_H(x)=\delta_G(x)=2$ para cada $x \in A$, esto implica que $\langle A \rangle_G$ es un ciclo. Si $G \not = \langle A \rangle_G $ no fuera un ciclo existir\'ia un $x \in A$ tal que $\delta(x)_G\geq 3$, lo cual es una contradicci\'on.

\item[Paso 2:] \emph{Se demuestra el teorema.}

Si $G$ es un ciclo, no es muy dif\'icil de verificar que $I=V$ \'o $I=S$. En otro caso, sea $I$ un conjunto independiente de $\S{G}$ tal que $I=A \cup B$ con $A \subseteq V$ y $B \subseteq S$. Por el Paso $1$, $|A|< |N_{\S{G}}(A)|$, y adem\'as $N_{\S{G}}(A)\cap B=\emptyset$. De esta forma $N_{\S{G}}(A)\cup B$ es un conjunto independiente estrictamente m\'as grande que $I$. 
\end{enumerate} 
\end{proof}

\begin{Thm}\label{matching}
Sea $G$  una gr\'afica de orden $n$ y tama\'no $m$, entonces $\beta(\S{G}) = min \lbrace n,m \rbrace$.
\end{Thm}
\begin{proof}
Supongamos que $\beta(\S{G})> min \lbrace m,n \rbrace$. Sea $M$ un matching m\'aximo de $\S{G}$, entonces $|M|>min \lbrace m,n \rbrace$, es decir, $|M|>n$ \'o $|M| >m$. Ya que cada arista $ev$ de $\S{G}$ cumple que $e \in S$ y $v \in V$, entonces en $|E(G)| >m$ \'o $|V(G)| >n$. Por lo tanto $\beta(\S{G})\leq min \lbrace n,m \rbrace$.
\begin{enumerate}
\item[Paso 1:]\emph{El resultado es cierto para arboles.}\\
Si $G$ es un \'arbol, se tiene que $min \lbrace m,n \rbrace=m$, entonces basta demostrar que existe un matching con $m$ elementos.\\
Procederemos a demostrar este paso usando inducci\'on sobre $n$.\\
Si $n=2$,  no es dif\'icil  verificar el resultado.\\
Sean $G$ un \'arbol con $n>2$, y $v \in V$ tal que $\delta(v)=1$. Notemos que existe $e \in S$ tal que $v \in N_{\S{G}}(e)$. Sea $G'=G \setminus \lbrace v \rbrace$, por hip\'otesis de inducci\'on $\beta(\S{G'})=min \lbrace m-1,n-1 \rbrace$, por lo tanto existe un matching m\'aximo $M$ en $\S{G'}$ tal que $|M|=m-1$. Obs\'ervese que $ev$ no es incidente con alguna arista de $M$ en $\S{G}$. Por lo tanto $M'=M\cup \lbrace ev \rbrace$ es un matching m\'aximo de $\S{G}$. 
\item[Paso 2:]\emph{Para cada $v \in V(G)$, existe un matching m\'aximo $M$ tal que $v$ no es extremo de ninguna arista de $M$.}\\
Una vez m\'as, usaremos inducci\'on sobre $n$ siendo el paso base para $n=2$, no muy dif\'icil de verificar.\\
Sean $G$ un \'arbol y $v \in V$, con $N_G(v)=\lbrace v_1, v_2, \ldots, v_k \rbrace$ y $N_{\S{G}}(v)=\lbrace e_1, e_2, \ldots, e_k \rbrace$. Notemos que la gr\'afica $\S{G\setminus \lbrace v \rbrace}$ tiene $k$ componentes conexas $G_i$ ($i=1,\ldots,k$), adem\'as $|V(G_i)|=n_i<n$ para todo $i=1,\ldots,k$. Por el Paso $1$ y por la hip\'otesis de inducci\'on existe un matching m\'aximo $M_i$ de $G_i$ con $n_i-1$ aristas de tal manera que $v_i$ no es extremo de ninguna arista de $M_i$. Obs\'ervese que cada arista de $A=\lbrace v_1e_1, v_2e_2,\ldots, v_ke_k \rbrace$ no es incidente con alguna arista de $M=\cup_{i=1}^{k}  M_i$. As\'i $M'=A \cup M$ es un matching m\'aximo de $\S{G}$. \\
\item[Paso 3:]\emph{El resultado es cierto en general.}\\
Sean $G \not \cong T$, $T_G$ un \'arbol generador de $G$ y $v$, $v'$ dos v\'ertices de $G$ de tal manera que $T_G=G \setminus \lbrace vv'\rbrace$ con $vv' \in E(G)$. Por definici\'on de $\S{G}$, existe $e \in S$ tal que $v,v' \in N_{\S{G}}(e)$. Por el Paso $1$, $\S{T_G}$ contiene un matching m\'aximo $M$ con $n-1$ aristas. Adem\'as, por el Paso $2$, sin p\'erdida de generalidad, $v$ no es extremo de ninguna arista de $M$. Obs\'ervese que $ev$ no es incidente con alguna arista de $M$ en $\S{G}$. Por lo tanto $M'=M\cup \lbrace ev \rbrace$ es un matching m\'aximo de $\S{G}$. 
\end{enumerate}
\end{proof}

\begin{Cor}
Sea $G$ es una gr\'afica de orden $n$ y tama\~no $m$, entonces $\rho(\S{G})=m\acute{a}x\lbrace n,m \rbrace$.
\end{Cor}

\begin{Cor}
Sea $G$ una gr\'afica, entonces
\begin{enumerate}
\item[i)] $\alpha(\S{G})=\rho(\S{G})$.
\item[ii)] $\tau(\S{G})=\beta(\S{G})$.
\end{enumerate}
\end{Cor}

\begin{Thm}
Sea $G$ una gr\'afica. Los siguientes enunciados son equivalentes.
\begin{itemize}
\item[i)] $G$ contiene a lo m\'as un ciclo.
\item[ii)] $\beta(\S{G})=|S|$.
\item[iii)] $\rho(\S{G})=|V|$.
\end{itemize}
\end{Thm}
\begin{proof}
Vamos a demostrar que i) $\Rightarrow$ ii) $\Rightarrow$ i) y que ii) $\Rightarrow$ iii) $\Rightarrow$ ii).

i) $\Rightarrow$ ii). Supongamos que $G$ contiene a lo m\'as un ciclo, es decir $G$ no tiene ciclos \'o $G$ tiene un \'unico ciclo. Si $G$ no tiene ciclos, entonces $G$ es un \'arbol, y por lo tanto $|S|=|E(G)|=|V(G)|-1<|V(G)|=|V|$. Si $G$ contiene un \'unico $k$-ciclo, entonces $|V(G)|=|E(G)|$, en cualquier caso, el resultado se desprende del Teorema \ref{matching}.

ii) $\Rightarrow$ i). Supongamos que $\beta(\S{G})=|S|$. Por el Teorema \ref{matching}, $\beta(\S{G})=min \lbrace |V|,|S|\rbrace$, por lo tanto $|E(G)|=|S|\leq|V|=|V(G)|$. Por el Lema \ref{mio} $G$ contiene a lo m\'as un ciclo.

Las implicaciones ii) $\Rightarrow$ iii) y iii) $\Rightarrow$ ii) se desprenden del Teorema \ref{Galai}.
\end{proof}

\begin{Cor}
Sea $G$ una gr\'afica. Los siguientes enunciados son equivalentes.
\begin{itemize}

\item[i)] $G$ contiene al menos un ciclo.
\item[ii)] $\beta(\S{G})=|V|$.
\item[iii)] $\rho(\S{G})=|S|$.
\end{itemize}
\end{Cor}

\begin{Cor}
Sea $G$ una gr\'afica conexa simple de orden $n$ y tama\~no $m$, entonces $|V(\S{G})|=\alpha(\S{G})+\beta(\S{G})$.
\end{Cor}

\begin{Prop}
Una gr\'afica $G$ es simple si y s\'olo si $\S{G}$ es simple y no contiene subgr\'aficas isomorfas  a $C_4$.
\end{Prop}
\begin{proof}
Sup\'ongase que $\S{G}$ contiene un $4-ciclo$ $v_1v_2v_3v_4$. Como cada arista en $\S{G}$ tiene un extremo en $S$ y otro en $V$ podemos suponer sin p\'erdida de generalidad que $v_2,v_4 \in S$ mientras que $v_1,v_3 \in V$. Dado que si dos v\'ertices en $V$ tienen un vecino en com\'un en $S$ ello implica que dichos v\'ertices forman una arista en $G$. Al ser $v_2,v_4 \in N(v_1)\cap N(v_3)$ y $v_2 \not = v_4$ se concluye que en $G$ hay dos aristas distintas entre los v\'ertices $v_1,v_3$ y as\'i $G$ no es simple. Ahora supongamos que $\S{G}$ no es simple, entonces $\S{G}$ contiene aristas m\'ultiples o lazos. Observemos que $\S{G}$ no puede tener lazos ya que para cada arista de $\S{G}$ se cumple que uno de sus extremos est\'a en $V$ y el otro est\'a en $S$. Si $\S{G}$ tiene aristas m\'ultiples, entonces existen dos aristas distintas las cuales tienen por extremos a $v \in V$ y $u \in S$, por lo que $G$ tiene al menos un lazo y $G$ no es simple.   

Ahora supongamos que $G$ es una gr\'afica no simple, es decir, que tiene aristas m\'ultiples y/o lazos.
\begin{enumerate}
\item[Caso 1:] Si $G$ tiene dos aristas $e_1$ y $e_2$ distintas tal que ambas tienen como v\'ertices extremos a $v, v' \in V(G)$, al aplicar la operaci\'on subdivisi\'on obtenemos que $ve_1v^\prime e_2$ es un $4-ciclo$ en $\S{G}$.
\item[Caso 2:] Si $G$ tiene un lazo, entonces existe una arista $e$ cuyo \'unico extremo es un v\'ertice $v$. Como $e$ es dos veces incidente a $v$, al aplicar la operaci\'on subdivisi\'on obtenemos que hay dos aristas distintas en $\S{G}$ con extremos $v$ y $e$.
\end{enumerate} 
\end{proof}

\begin{Prop}
Sea $G$ una gr\'afica simple, bipartita sin $4$-ciclos con bipartici\'on $X$, $Y$ en donde para cada $x \in X$, $\delta(x)=2$. Entonces $\alpha(G)=m\acute{a}x \lbrace |X|, |Y| \rbrace$. 
\end{Prop}
\begin{proof}
Sea $H$ la gr\'afica con conjunto de v\'ertices $Y$ y en donde para $y_1,y_2 \in Y$ son adyacentes si y s\'olo si existe $x \in X$ con $x \backsim y_1$ y $x \backsim y_2$, entonces $\S{H} \cong G$. 
\end{proof}

\begin{Prop}
Sea $G$ una gr\'afica. Si $D$ es un conjunto diferencial en $\S{G}$, entonces $D$ es un conjunto independiente en $\S{G}$.
\end{Prop}
\begin{proof}
Supongamos que $D$ no es un conjunto independiente en $\S{G}$. Entonces existen $v \in V$ y $e \in S$ tal que $v \backsim e$ en $D$. Por la Proposici\'on \ref{s1} iii), existe $v' \in V$ tal que $v' \backsim e$. Vamos a demostrar que $D$ no es un conjunto diferencial en $\S{G}$. Para ello consideremos los siguientes casos.

Caso 1: $v' \in D$. Supongamos que $\delta(v)\geq 2$ y $\delta(v')\geq 2$. Entonces $\partial(D\setminus \lbrace e \rbrace)=\partial(D)+2>\partial(D)$, y as\'i $D$ no es un conjunto diferencial en $\S{G}$. Sin p\'erdida de generalidad, supongamos que $\delta(v)=1$ y $\delta(v')\geq 2$. Entonces $\partial(D\setminus \lbrace e \rbrace)=\partial(D)+2>\partial(D)$ y as\'i $D$ no es un conjunto diferencial en $\S{G}$.

Caso 2: $v' \in B_{\S{G}}(D)$. Supongamos que $\delta(v)\geq 2$ y $\delta(v')\geq 2$. Entonces $\partial((D\setminus \lbrace e \rbrace) \cup \lbrace v' \rbrace)=\partial(D)+\delta(v')-1>\partial(D)$, y as\'i $D$ no es un conjunto diferencial en $\S{G}$. Ahora , supongamos que $\delta(v)=1$ y $\delta(v')\geq 2$. Entonces $\partial((D\setminus \lbrace e \rbrace)\cup \lbrace v'\rbrace)=\partial(D)+\delta(v')-1>\partial(D)$, y as\'i $D$ no es un conjunto diferencial en $\S{G}$. Por \'ultimo, supongamos que $\delta(v)\geq 2$ y $\delta(v')=1$. Entonces $\partial(S\setminus \lbrace e \rbrace)=\partial(D)+1>\partial(D)$, y as\'i $D$ no es un conjunto diferencial en $\S{G}$. 
\end{proof}

\begin{Cor}
Sea $G$ una gr\'afica. Si $D$ es un conjunto diferencial en $\S{G}$, entonces $|D|\leq \alpha(\S{G})$.
\end{Cor}

\begin{Prop}
Sea $G$ una gr\'afica. Entonces $\partial(G)=\partial(\S{G})$ si y s\'olo si $G \cong S_n$.
\end{Prop}
\begin{proof}
Supongamos que $\partial(G)=\partial(\S{G})$. Sea $D$ un conjunto diferencial de $G$ y de $\S{G}$. Entonces $|B_G(D)|=|B_{\S{G}}(D)|$. \\ 
Ahora, supongamos que $G \cong S_n$. Entonces $\partial(G)=n-2$. Por la Proposici\'on \ref{s3} $i)$, $\partial(\S{G})=n-2$. Por lo tanto $\partial(G)=\partial(\S{G})$.
\end{proof}

\begin{Lem}
Sea $G$ una gr\'afica. Para cualquier conjunto $D \subseteq V(G)$, $|B_{G}(D)|\leq |B_{\S{G}}(D)|$.
\end{Lem}
\begin{proof}
Sea $F=\lbrace e \in E(G) : e=uv, u \in D, v \in B_G(D) \rbrace$. Obs\'ervese que $F \subseteq B_{\S{G}}(D)$. Para cada $v \in B_{G}(D)$ elijase exactamente un v\'ertice $u_v \in D $ adyacente a $v$. La asignaci\'on $\varphi: B_{G}(D)\rightarrow F$ dada por $\varphi(v)=vu_v$ esta bien definida y es inyectiva. De esto \'ultimo se obtiene que $|B_{G}(D)|\leq |F|$. Por lo tanto se obtiene lo deseado.
\end{proof}

\begin{Lem}
Sea $G$ una gr\'afica. Para cualquier conjunto $D \subseteq V(\S{G})$, $|B_{\S{G}}(D)|\leq |B_{\Q{G}}(D)|$.
\end{Lem}
\begin{proof}
Paso $1$: \emph{El resultado es cierto cuando $D \subseteq V$.}

Paso $2$: \emph{El resultado es cierto cuando $D \subseteq S$.}

Paso $3$: \emph{El resultado es cierto en general.}

\end{proof}

\begin{Prop}
Para cualquier gr\'afica $G$, $\partial(G)\leq \partial(\S{G}) \leq \partial(\Q{G}) $. 
\end{Prop}
\begin{proof}
Sea $D$ un conjunto diferencial de $G$, entonces $$\partial(G)=\partial(D)=|B_G(D)|-|D|\leq |B_{\S{G}}(D)|-|D|=\partial_{\S{G}}(D)\leq \partial(\S{G}).$$
Ahora, sea $D'$ un conjunto diferencial de $\S{G}$, entonces $$\partial(\S{G})=\partial(D')=|B_{\S{G}}(D')|-|D'|\leq |B_{\Q{G}}(D')|-|D'|=\partial_{\Q{G}}(D')\leq \partial(\Q{G}).$$
\end{proof}

\begin{Cor}
Para cualquier gr\'afica $G$, $$\dfrac{\alpha(\S{G})+\alpha(G)}{2}\leq \dfrac{\alpha(\S{G})+\partial(G)}{2}\leq \partial(\S{G}).$$
\end{Cor}

\begin{Prop}
Sea $G$ una gr\'afica.
\begin{enumerate}
\item[i)] Si $D$ es un conjunto independiente en $G$, entonces $D$ es un conjunto independiente en $\S{G}$.
\item[ii)] Si $M$ es un matching en $G$, entonces $M$ es un conjunto independiente en $\S{G}$.
\end{enumerate}
\end{Prop}

\begin{Cor}\label{s2}
Sean $P_n$ y $C_n$ la gr\'afica camino y la gr\'afica ciclo respectivamente, entonces
\begin{enumerate}
\item[i)] $\partial(\S{P_n})=\left\lfloor \frac{2n-1}{3} \right\rfloor$ con $n \geq 2$.
\item[ii)] $\partial(\S{C_n})=\left\lfloor \frac{2n}{3} \right\rfloor$ con $n\geq 3$.
\end{enumerate}
\end{Cor}
\begin{proof}
Observemos que $\S{P_n}=P_{2n-1}$. Recordemos que $\partial(P_n)=\left\lfloor \frac{n}{3} \right\rfloor$, por lo tanto
\[
\partial(\S{P_n})=\partial(P_{2n-1})=\left \lfloor \frac{2n-1}{3} \right \rfloor.
\] 
Para ii) la demostraci\'on es an\'aloga solo basta observar que $\S{C_n}=C_{2n}$ y recordar que $\partial(C_n)=\lfloor \frac{n}{3} \rfloor$.  
\end{proof}

\begin{Prop}\label{s3}
Sean $S_n$, $S_{m,n}$, $W_n$ y $K_n$ las gr\'aficas estrella, doble estrella, rueda, completa y bipartita completa respectivamente, entonces
\begin{enumerate}
\item[i)] $\partial(\S{S_n})=\partial(S_n)=n-2$ con $n\geq 3$. 
\item[ii)] $\partial(\S{S_{p,q}})=\partial(S_{p,q})+1=p+q-3$.
\item[iii)] $\partial(\S{W_n})=\partial(S_{n})+\partial(C_{2(n-1)})=(n-2)+\lfloor \frac{2(n-1)}{3} \rfloor$.
\item[iv)] $\partial(\S{K_n})=\dfrac{n(n-1)}{2}+2-n$.
\item[v)]$\partial(\S{K_{p,q}})=p(q-1)$, con $p<q$.
\end{enumerate}
\end{Prop}
\begin{proof}
\begin{enumerate}
\item[i)] Dado $D \subseteq V(\S{S_n})$, supongamos primero que el v\'ertice \'apice $v_0$ de $S_n$ no est\'a en $D$ y sean $x$ el n\'umero de v\'ertices adyacentes a $v_0$ que se encuentran en $D$ y que no son adyacentes a ning\'un v\'ertice de $D$, $y$ el n\'umero de hojas de $\S{S_n}$ que se encuentran en $D$ y no son adyacentes a ning\'un v\'ertice de $D$, mientras que $z$ es la cantidad de aristas de $\S{S_n}$ cuyos extremos se encuentran en $D$, donde $x\geq 0$, $y\geq 0$ y $z>0$ (see Figure \ref{Fig:S(S_n)}). Entonces $$ \partial(D)=1+x+y-(x+y+2z)=1-2z, $$ y esta funci\'on alcanza su m\'aximo cuando $z=1$, obteniendo como diferencial $-1$.\\
Si $z=0$ entonces $\partial(D)=1+x+y-(x+y)=1$.\\
Supongamos ahora que el \'apice $v_0$ se encuentra en $D$. Sean $x$, $y$ y $z$ como al principio. En este caso el diferencial toma la siguiente forma 
$$\partial(D)=x+y-(x+y+2z+1)=-2z-1$$ y esta funci\'on alcanza su m\'aximo cuando $z=1$, obteniendo como diferencial $-3$.
Si $z=0$ entonces $\partial(D)=x+y+(n-1-x-y)-x-y-1=n-2-x-y$ y esta funci\'on toma su m\'aximo en $x$ y $y$ iguales a $0$, y el valor m\'aximo es $n-2$.\\
Por lo tanto $\partial(\S{S_n})=n-2$. 
\item[ii)] Sea $D$ un conjunto diferencial de $\S{S_{p,q}}$. Si $|D|=2$, entonces $\partial(D)=(p-1)+(q-1)+1-2=p+q-3$. Si $|D|=1$, entonces $\partial(D)=m\acute{a}x\lbrace p,q \rbrace$. Si $|D|\geq 3$, entonces $\partial(D)=|B(D)|-|D|=(p-1)+(q-1)-(|D|-2)+(|D|-2)+1-|D|=p+q-1-|D|\leq p+q-1-3=p+q-4$, por lo tanto $\partial(\S{S_{p,q}})=p+q-3$.

\item[iv)] Sea $K_n$ la gr\'afica completa de orden $n$ con conjunto de v\'ertices $V(K_n)=\lbrace v_1, v_2, \ldots, v_n  \rbrace$. Definamos el conjunto $D=\lbrace v_1, v_2, \ldots , v_{n-2}, u \rbrace$ donde $v_i \in V$ ($i=1,2,\ldots, n-2$) y $u \in S$. Calcularemos $\partial(D)$ usando $iii)$ y $iv)$ de la Proposici\'on \ref{s1}. Entonces $$\partial(D)=\dfrac{n(n-1)}{2}+2-n. $$
Consideremos $D' \subset V(\S{G})$. Vamos a demostrar que $\partial(D)\geq \partial(D')$. Sean $D_1 \subseteq V$, $D_2 \subset S$ tal que $D'=D_1\cup D_2$, entonces 
\begin{enumerate}
\item[Caso 1:] $|D_1|>n-2$.

Si $|D_1|=n-1$, $\partial(D') \leq \frac{n(n-1)}{2}+2-n-2|D_2|$. Para $|D_2|=0$, $\partial(D')=\frac{n(n-1)}{2}+1-n<\partial(D)$. Para $|D_2|\geq 1$, $\partial(D')\leq \frac{n(n-1)}{2}+2-n-2|D_2|\leq \frac{n(n-1)}{2}-n< \partial(D).$

Si $|D_1|=n$, $\partial(D')\leq \frac{n(n-1)}{2}-n-2|D_2|$. Para $|D_2|=0$, $\partial(D')= \frac{n(n-1)}{2}-n<\partial(D)$. Para $|D_2|\geq 1$, $\partial(D')\leq \frac{n(n-1)}{2}-n-2|D_2|\leq \frac{n(n-1)}{2}-n-2<\partial(D)$.

\item[Caso 2:] $|D_1|<n-2$.

Si $|D_1|\leq n-3$, $\partial(D')\leq \frac{n(n-1)}{2}+6-n-2|D_2|$. Para $|D_2|=0$, $\partial(D')\leq \frac{n(n-1)}{2}-n<\partial(D)$. Para $|D_2|= 1$, $\partial(D')\leq \frac{n(n-1)}{2}+1-n<\partial(D)$. Para $|D_2|=2$, $\partial(D')\leq \frac{n(n-1)}{2}+1-n<\partial(D)$. Para $|D_2|\geq 3$, $\partial(D')\leq \frac{n(n-1)}{2}+6-n-2|D_2|\leq\frac{n(n-1)}{2}+6-n-6<\partial(D)$.

\item[Caso 3:] $|D_1|=n-2$, $\partial(D')\leq \frac{n(n-1)}{2}+4-n-2|D_2|$. Si $|D_2|=0$, $\partial(D')=\frac{n(n-1)}{2}+1-n<\partial(D)$. Si $|D_2|=1$, $\partial(D')=\partial(D)$. Si $|D_2|\geq 2$, $\partial(D')\leq\frac{n(n-1)}{2}+4-n-2|D_2|\leq\frac{n(n-1)}{2}+4-n-4<\partial(D)$.
\end{enumerate}
Por lo tanto $\partial(D)\geq \partial(D')$ y as\'i $\partial(\S{G})=\partial(D)$.
\end{enumerate}
\end{proof}

\begin{figure}
\begin{center}
\includegraphics[width=6cm]{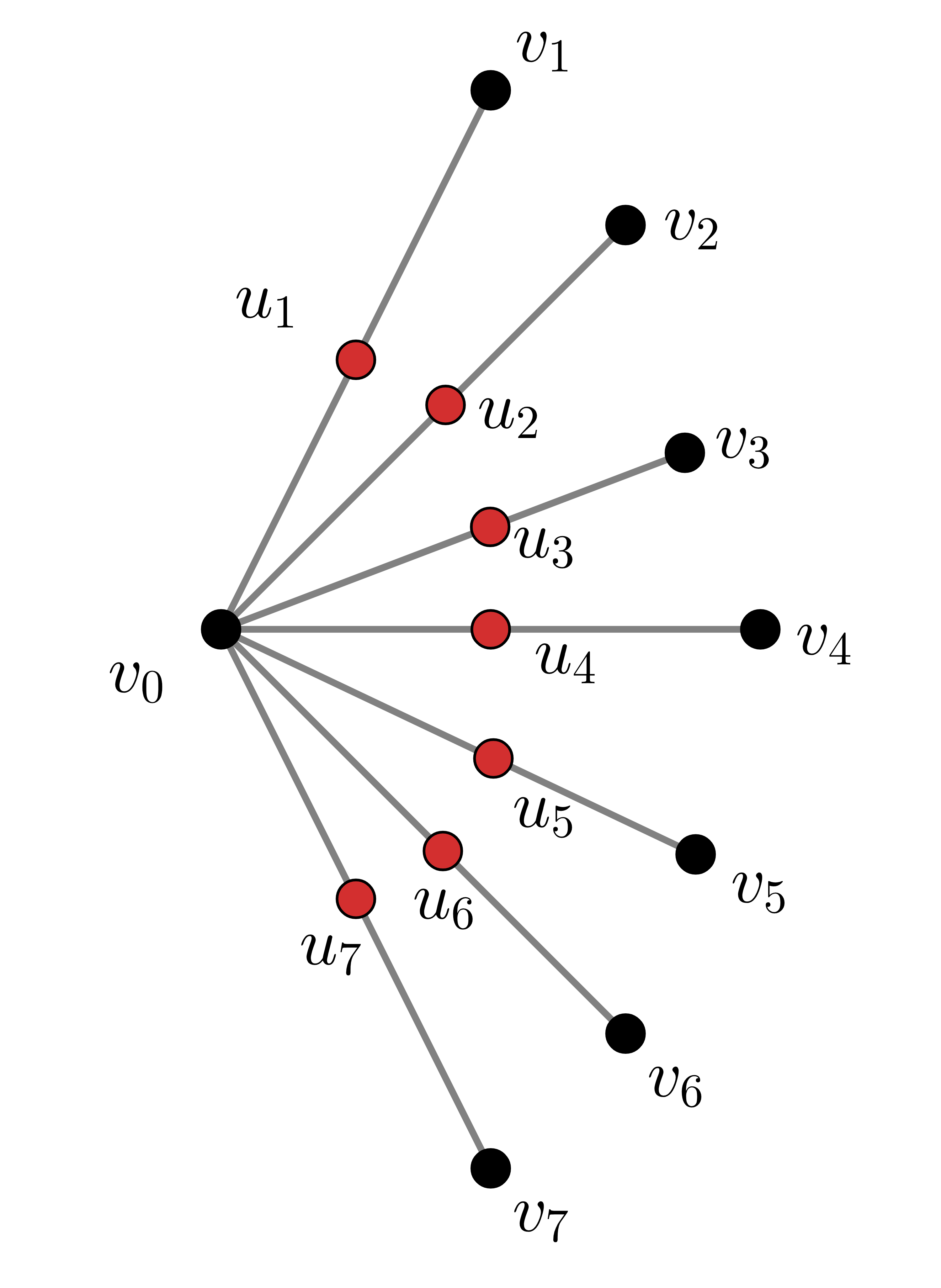} 
\end{center}
\caption{Para $\S{S_8}$, si $X=\lbrace u_1,u_2,u_3 \rbrace$, $Y=\lbrace v_6,v_7 \rbrace$ y $Z=\lbrace u_4,u_5,v_4,v_5 \rbrace$ tal que $D=X\cup Y\cup Z$, entonces  $x=3$, $y=2$ y $z=4$.}
\label{Fig:S(S_n)}
\end{figure}

\section{Acknowledgements}


\end{document}